\documentclass[12pt]{amsart}
\usepackage{graphicx}
\usepackage[all]{xy}
\usepackage{amscd}
\usepackage{amssymb}
\usepackage{amsmath}
\usepackage{amsthm}
\usepackage{amsfonts}
\usepackage{graphics}
\usepackage{epsfig,psfrag}
\usepackage[english]{babel}
\usepackage{latexsym}
\usepackage{mathrsfs}

\newtheorem{theorem}{Theorem}[section]

\newtheorem{lemma}[theorem]{Lemma}

\newtheorem{proposition}[theorem]{Proposition}
\newtheorem{question}[theorem]{Question}

\theoremstyle{definition}

\newcommand{\RR}{{\mathbb R}}

\newcommand{\CC}{{\mathbb C}}
\newcommand{\ZZ}{{\mathbb Z}}
\newcommand{\FF}{{\mathbb F}}

\newcommand{\Arf}{{\mathrm{Arf}}}

\newcommand{\Spin}{{\rm Spin}}

\theoremstyle{remark}
\newtheorem{remark}[theorem]{Remark}

\begin{document}

\title[Extending periodic maps on surfaces over  the 4-sphere]
{Extending periodic maps on surfaces over  the 4-sphere}

\author{Shicheng Wang}
\address{Department of Mathematics, Peking University, Beijing 100871, China}
\email{wangsc@math.pku.edu.cn}

\author{Zhongzi Wang}
\address{Department of Mathematics, Peking University, Beijing 100871, China}
\email{2201110025@stu.pku.edu.cn}



%\thanks{}
\begin{abstract}
Let $F_g$ be the closed  orientable surface of genus $g$.
We address the problem to extend torsion elements of the mapping class group ${\mathcal{M}}(F_g)$ over the 4-sphere $S^4$. 
Let  $w_g$  be a  torsion element of maximum order in ${\mathcal{M}}(F_g)$. 
Results including:

(1) For each $g$,
$w_g$ is periodically  extendable over $S^4$ for some non-smooth embedding  $e: F_g\to S^4$,
and  not periodically extendable over $S^4$ for any smooth embedding  $e: F_g\to S^4$.

(2) For each $g$, $w_g$  is extendable over $S^4$ for some smooth embedding  $e: F_g\to S^4$
if and only if $g=4k, 4k+3$.

(3)  Each torsion element of order $p$ in ${\mathcal{M}}(F_g)$ is extendable over $S^4$ for some smooth embedding  $e: F_g\to S^4$ if 
 either (i)  $p=3^m$ and $g$ is even;
 or (ii) $p=5^m$ and $g\ne 4k+2$;
 or (iii) $p=7^m$.
 Moreover  the conditions on $g$ in (i) and (ii) can not be removed .  
\end{abstract}

\date{}
\maketitle
\tableofcontents
\section{Introduction}

In this paper, we will work on the smooth, orientable category unless otherwise clearly stated.  
We use $Aut(M)$ to denote the orientation-preserving diffeomorphism group of a manifold $M$, and 
$|A|$ to denote the cardinality of a finite set $A$. 
Let $\RR^n$ and $S^n$ be the $n$-space and $n$-sphere respectively.
Let $F_g$ be the closed orientable surface of genus $g>0$. 
Let ${\mathcal{M}}(F_g)$ be the mapping
class group of $F_g$, i.e., the group of
smooth isotopy classes of orientation preserving
diffeomorphisms on $F_g$. 

An element $f$ in $Aut(F_g)$ is {\it extendable over $M$ with respect to an embedding $e:F_g\rightarrow M$} if there exists an element $\tilde f$ in $Aut(M)$ such that $$\tilde f\circ e=e\circ f.$$
When $f$ is a map of finite order, we say {\it $f$ is {\it periodically extendable} over $M$ for $e$}, if $\tilde f$ above is a map of the same order. Call an element $f$ in $Aut(F_g)$ is {\it (periodically) extendable over $M$}, if  $f$ is  (periodically) extendable over $M$ for some embedding $e$.

In the definition above, if we allow $\tilde f$ to be just a homeomorphism on $M$, then we call $f$ is {\it topologically extendable}.

%either $F_g\subset M$ and the embedding is refer to the inclusion map,
%or we omit "w.r.t.$e$". 

We address the problem when torsion elements of $Aut(F_g)$ are extendable, or periodically extendable  over $S^4$.
There are two motivations of this study.

One motivation is rather naive:  In order to get intuitive feelings of symmetries on surfaces,  people try to embed those surfaces symmetries into the symmetries
of our 3-space (3-sphere) in which we live.
 Some  recent work in this direction include 
 \cite{WWZZ}, \cite{NWW}. 
The first one tell us that  each periodic map on $F_g$ of order $k>2g+2$ are not  periodically
extendable over $S^3$ for any embedding; the second  one claims  if a periodic map $f$ on $F_g$ is extendable  over $S^3$ for some embedding $e: F_g\to S^3$, then $f$ is periodically extendable over $S^3$ for some embedding $e' : F_g\to S^3$.  Therefore  each map of order $k>2g+2$ is not
extendable  over $S^3$ for any embeddings.

Another motivation is to  extend surface automorphisms over $S^4$ (or  $\RR^4$): In 1983, for a trivial smooth
embedding of the torus  $F_1$ in $S^4$,  Montesinos proved that $f\in  Aut(F_1)$ is extendable over $S^4$
if and only if $f$ preserves the induced Rohlin form \cite{Mon}.  In 2002 the same criterion
is proved by Hirose for any $g$  in a comprehensive work \cite{Hi}. 
In 2012 it is proved if $f\in Aut(N^n)$ is extendable (or topologically extendable) over $\RR^{n+2}$  for any co-dimension 2
smooth embedding $e: N^n\to \RR^{n+2}$, then $f$ preserves  the  induced spin structures \cite{DLWY}. 
In the case of $F_g\subset S^4$ (or $\RR^4$), the induced Rohlin forms and induced spin structures coincide (Lemma \ref{Rohlin}).
To apply above results to see if an element in ${\mathcal{M}}(F_g)$ is extendable, we may first to try a simpler class, the torsion elements, specially  those with either big orders or prime orders, due geometric and algebraic reasons.

Recall that the largest and the second largest orders of  torsions  in ${\mathcal{M}}(F_g)$ are $4g+2$  and $4g$.
Torsion elements of order $4g+2$ are called Wiman maps.

\begin{theorem}\label{group-ext} Let $w_g$ be a Wiman map on $F_g$.

(1)  Then $w_g$ is periodically  extendable over $S^4$ for some non-smooth embedding  $e: F_g\to S^4$ for each $g$;

(2) $w_g$  is not periodically extendable over $S^4$ for any smooth embedding  $e: F_g\to S^4$ for each $g$.
\end{theorem}

\begin{theorem}\label{4g+2,4g} 

(1)  Let $w_g$ be a Wiman map on  $F_g$. Then  $w_g$  is extendable over $S^4$ for some smooth embedding  $e: F_g\to S^4$
if and only if $g=4k, 4k+3$.

(2) Let $v_g$ be a map of order $4g$ on  $F_g$. Then  $v_g$  is extendable over $S^4$ for some smooth embedding  $e: F_g\to S^4$
if and only if $g=4k, 4k+1, 4k+3$.
\end{theorem}
 
 \begin{theorem}\label{3-5-7}
 Each periodic map  $f: F_g\to F_g$  of order $p$ is extendable over $S^4$ if 
 
 (i)  $p=3^m$ and $g$ is even;
 
 (ii) $p=5^m$ and $g=4k, \, 4k+1, \, 4k+3$;
 
 (iii) $p=7^m$, where 7 can be replaced by infinitely many other primes. 
 
 Moreover,  the conditions on $g$ in (i) and (ii) can not be removed.  
 \end{theorem}  
 
\begin{remark} There are two topological versions of extending  surface automorphisms over the 4-sphere:
The first one  allows those embeddings $F_g \to S^4$ non-smooth (topological embeddings), and 
the second one allows those extensions over $S^4$ to be homeomorphisms (topological extensions).

(1) By comparing Theorem \ref{group-ext} (1) and Theorem \ref{4g+2,4g} (1), the results about smooth embedding and non-smooth
embedding are quite different. Note that in Theorem \ref{group-ext} (1),  the non-smooth embedding   is piecewise smooth, and  the element extending $w_g$ is smooth, indeed in $SO(4)$.  

(2) On the other hand, assume that the embeddings $F_g\to S^4$ are smooth, then the results for smooth extensions and topological extensions are the same, in particular Theorem \ref{4g+2,4g} still holds for topological extensions. This follows from Theorem \ref{criterion}.

(3) To compare with the result in $S^3$ \cite{NWW}, Theorem \ref{group-ext} (2)  and Theorem \ref{4g+2,4g} (1) provide  periodic maps on surfaces which are extendable over $S^4$ for some smooth embedding, but not 
periodically extendable  over $S^4$ for any smooth embeddings.

(4) For any $g>1$, there exists an $f_g\in Aut(F_g)$ which is not extendable   over $\RR^4$ for any embedding $e$ by \cite{DLWY}.
Explicit $f_g$ is unknown before.
With Theorem \ref{4g+2,4g}, we wonder if such $f_g$ can be a torsion element for each $g$.

%Theorem \ref{diff-ext} is not expected by the authors. It is interesting to know 
%that punctured lens spaces $L(p, q)$ are embeddable into $S^4$ for odd $p$ by Zeeman \cite{Zee}, and  not embeddable into $S^4$ for even $p$ by Epstein \cite{Ep}.

%(3) Let $\mathcal T_g$ be the finite index subgroup of  ${\mathcal{M}}(F_g)$ which acts trivially on 
%$H_1(F_g, Z_2)$. 
%For $g=4k+1$, $4k+2$, $w_g\circ h$ is  not extendable over $S^4$ 
%for any embeddings, where $w_g$ is a Wiman map, and $[h]\in \mathcal T_g$.

\end{remark}

\begin{remark} %A result of Whitney claims that every  $m$-dimensional  manifold $M^m$ can be embedded into $\RR^{2m}$
%\cite{Wh}.
%and to search the smaller $n$ for concrete  3- and 4-manifolds $M$  so that $M$ can be embedded into $\RR^{n}$ is a difficult topic 
%in manifold topology. 
A result of Mostow, also Palais,  claims that any compact Lie group action on $M^m$ can be embedded to the symmetry of  $\RR^n$ for some $n$, \cite{Mos}, \cite{Pal}. %Correspondingly o
One may ask:  

\begin{question}
For a given finite order map $f\in Aut(F_g)$,  what are possibly small $n$ so that $f$ is periodically extendable over
$S^n$ (or $\RR^n$)?
\end{question}

Some explicit bounds of above question is given in \cite{WaZ}.
\end{remark}

The organization of the paper is reflected in the table of content. The proof of Theorem \ref{group-ext} given in Section 4 can be read directly.

{\bf Acknowledgement:} The paper is revised based  the referee's advise and comments: 
We make a  general setting to extend torsion elements of ${\mathcal{M}}(F_g)$ over $S^4$ rather than just to work on Wiman maps in the old version, and  we  add some discussion of topological extension. Toward  those directions, we state explicitly an extending criterion (Theorem \ref{criterion}), and get some new results (Theorem \ref{4g+2,4g} (2) and Theorem \ref{3-5-7}).

We thank the referee for the insightful advise and comments. 
We thank Fan Ding, Yi Liu and Chao Wang 
for mathematics they taught us.  We thank Hongbin Sun for Lemma \ref{many primes}.

\section{Spin structures}

All facts related spin structures below  can be found in pages 33-37 and
page 102 \cite{Ki},  \cite{Roh}, and \cite{DLWY}.

%There are several different descriptions of spin structures.

Spin structures of a rank $n$ vector bundle $\xi$ over a CW complex
$X$ can phrased with trivialization. $\xi$ can be
endowed with a spin structure if $\xi\oplus\epsilon^k$ has a
trivialization over the $1$-skeleton of $X$ which may extend over the
$2$-skeleton of $X$, (if $n\geq 3$, $k=0$; if $n=2$, $k=1$; if $n=1$,
$k=2$), and a spin structure is a homotopy class of such
trivializations. 
%For a CW complex $X$, we use
%$X^{(i)}$ to denote its $i$-skeleton.
%For a rank $n$ vector bundle $\xi$ over a CW space $M$, a spin
%structure on $\xi$ is known as with respect to any CW structure on
%$M$, up to the natural equivalence. 
A spin structure of a smooth
manifold $M$ is a spin structure of its tangent bundle with respect to some (hence any) CW complex structure of $M$, which exists
if and only $M$ is orientable and the second Stiefel-Whitney class $w_2(M)=0$.
%For any
%closed path $\alpha$ on $M$, we may also restrict a spin structure
%$\sigma$ of $M$ to $\sigma|_\alpha$, namely picking a trivialization
%of $E|_\alpha$ which extends to a trivialization on (some)
%$1$-skeleton equivalent to $\sigma$. 
Two spin structures $\sigma_1,\sigma_0$ of $\xi$ are equivalent
if and only if the trivialization $\sigma_1|_\alpha\simeq \sigma_0|_\alpha$ for
any closed path $\alpha$ on $M$. 
%For an oriented smooth manifold $M$, $M$ has a spin structure if and only if $w_2(M)=0$, and $M$ has
%spin structures, the space $\mathcal{S}(M)$ of spin structures on $M$ is an affine $H^1(M;\ZZ_2)$.
Two facts below are basic for us:

(i) The circle $S^1$ has two spin structures: the bounded
spin structure, namely which bounds  a spin structure on a surface $F$, or equivalently, the corresponding element in $\pi_1(SO(n))=\ZZ_2$ is trivial for some $n\ge 3$; 
and the Lie group spin structure, which is unbounded, or equivalently, the corresponding element in $\pi_1(SO(n))$ is non-trivial.

(ii) If $\xi=\xi_1\oplus \xi_2$ for bundles over $X$, then
spin structures on two of the bundles determine a spin structure of
the third (cf. \cite[p. 37]{Ki}).

For a spin manifold $M$ with boundary $\partial M$, $\partial M$ has a natural spin structure induced from the spin structure of $M$ and the (inward) normal vector of $\partial M$ in $M$ by (ii). A spin manifold is said to be spin bounded 
%(null spin cobordent), 
if there is a spin manifolds bounded by it, inducing its spin structure. 

\subsection{Spin structure on $F_g$} %Below we are concentrated on the spin structures on $F_g$. 
Let $\mathcal{S}(F_g)$ be the space of spin structures on $F_g$. It is known that spin structure on $F_g$ are 1-1 corresponding to elements in
$H^1(F_g,\ZZ_2)$, therefore there are $2^{2g}$ spin structures on $F_g$.

There is a surjective map:
$$\mathcal{S}(F_g)\stackrel{[.]}\longrightarrow \Omega^\Spin_2\stackrel\Arf\longrightarrow\ZZ_2,$$
where $\Omega^\Spin_2$ is the second spin cobordism group and $\Arf$ is the Arf isomorphism. More precisely,
for any $\sigma\in\mathcal{S}(F_g)$, there is an associated nonsingular quadratic function $q_\sigma:H_1(F_g;\ZZ_2)\to\ZZ_2$,
such that $q_\sigma(\alpha)=0$ (resp. $1$) if the spin
structure on $F_g$ restricted to the bounding (resp. Lie-group) spin structure on $\alpha$. Note
$$q_\sigma(\alpha+\beta)=q_\sigma(\alpha)+q_\sigma(\beta)+\alpha\cdot\beta \qquad (S1),$$
 where $\alpha\cdot\beta$ is the $\ZZ_2$-intersection number, and $\sigma=\sigma'$ if and only if $q_{\sigma}=q_{\sigma'}$. 
 Thus $\Arf([\sigma])$ is defined as the Arf invariant of the nonsingular quadratic form $q_\sigma$. 
 
Let $\mathcal{B}_g$ be the union of  bounding ($\Arf([\sigma])=0$)  spin structures
and $\mathcal{U}_g$ be the union of  unbounding ($\Arf([\sigma])=1$)  spin structures. Then $\mathcal{S}(F_g)$ is a disjoint union:
$$\mathcal{S}(F_g)=\mathcal{B}_g\sqcup\mathcal{U}_g.$$

\begin{lemma}  \label{BU}\cite[page 102]{Ki} Suppose $q_\sigma\in \mathcal S_g$. Then $q_\sigma\in \mathcal{B}_g$  (resp. $q_\sigma\in \mathcal{U}_g$) if and only if $q$ vanishes on exactly $2^{2g-1}+2^{g-1}$ (resp. $2^{2g-1}-2^{g-1}$) elements.
\end{lemma}

For any $f\in Aut(F_g)$, we have induced map $$f_*: H_1(F_g, \ZZ_2)\to H_1(F_g, \ZZ_2),$$  Furthermore we have
the map 
$$f^*: \mathcal{S}(F_g) \to \mathcal{S}(F_g)$$
given by 
$$f^*(\sigma)(x)=\sigma(f_*(x)) \qquad (S2)$$
for any $\sigma \in \mathcal{S}(F_g)$ and $x\in H_1(F_g, \ZZ_2)$, and 
for any $f,g \in Aut (F_g)$, 
$$(f\circ g)^*= g^*\circ f^* \qquad (S3).$$

Clearly (S2) defines the action  of $Aut (F_g)$ on $\mathcal{S}(F_g)$, which induces an actions of $\mathcal{M}(F_g)$
on $\mathcal{S}(F_g)$.

\begin{lemma}\label{trans}  \cite{Jo} \cite{DLWY} $\mathcal{M}(F_g)$ acts transitively on
$\mathcal{B}_g$ and on $\mathcal{U}_g$.
\end{lemma}

\begin{lemma}\label{cardinality}  \cite{DLWY} 
$|\mathcal{B}_g|= 2^{2g-1}+2^{g-1}$ and  $|\mathcal{U}_g|= 2^{2g-1}-2^{g-1}$.
\end{lemma}

\subsection{Induced Rohlin forms and induced spin structures}

Now we fix a point $x\in S^4$ and identify $\RR^4$ with $S^4\setminus x$.
Suppose $e:  F_g \hookrightarrow\RR^{4}=S^4\setminus x\subset S^4$ is a smooth embedding. 
We denote $e(F_g)$ by $F$. 

We first recall Rohlin form \cite{Roh}.
The Rohlin 
form $q_e: H_1(F_g;\ZZ_2)\to \ZZ_2$  for the embedding $e$ is defined so that for any smoothly embedded
subsurface $P\subset S^4$ with $\partial P\subset F$,
$\mathring{P}\subset S^4\setminus F$ and
 transverse to $F$ along $\partial P$, $q_e ([\partial P])$ is the $\bmod\,2$ number
 of points in $P\cap P'$, where $P'$ is a smooth perturbed copy of $P$ so that $\partial P'\subset F$
 is disjoint parallel to $\partial P$, and that $\mathring{P}'$ is transverse to $\mathring{P}$ . It is known that 
  $q_e$ is a bounded spin structure \cite{Roh}.
  
A handlebody $H_g$ is obtained by attaching $g$ 3-dimensional 1-handles to the 3-ball. %(see \cite{He}).
Call a smooth embedding of $e: F_g\to S^4$ is trivial, if $F=e(F_g)$  bounds a handlebody $H_g$ in $S^4$.   The image of trivial embedding $F \subset S^4$ is unique up to equivalence. 

\begin{theorem}\label{spin-obstruction2} (\cite{Hi}, \text{cf. also \cite{Mon} for $g=1$})
For a trivial embedding  $e: F_g\to S^{4}$ and $f\in Aut(F_g)$,  $f$ is extendable over $S^4$ if and only if $f$ preserves the induced  Rohlin form $q_e$.
\end{theorem} 

%We will describe the Rohlin form and the induced spin structure for this embedding.  

%For a spin manifold $M$ 
%with boundary $\partial M$, $\partial M$ has a natural spin 
%structure induced from the spin structure of $M$ 
%and the (inward) normal vector of $\partial M$ in $M$. A manifold is called a \emph{spin
%boundary} if there is a spin manifold bounding it, inducing its spin structure.

%For example, the circle $S^1$ has two spin structures: one spin-bounds the spin $D^2$, 
%and the other is the Lie-group spin structure which is not a spin-boundary.
Next we recall the constructions of induced spin structure \cite{DLWY}.
For each smooth embedding $F\subset \RR^4$,
two classical facts are: (i) the normal bundle $N(F)$ of $F$ in
$\RR^{4}$ is trivial; (ii)  there exists
a compact connected oriented $3$-dimensional smooth submanifold  $\Sigma\subset\RR^{4}$
such that $\partial\Sigma=F$. Let $W$ be an inward normal vector field of $F$ in $\Sigma$, 
%(say, w.r.t some compatible Riemannian metric on a collar), 
and $H$ be a normal vector field of $\Sigma$ in $\RR^{4}$ over $F$, such that the orientation $(W,H)$ of the normal bundle $N(F)$ and the orientation
of $F$ match up to that the canonical orientation of $\RR^{4}$. 

%Our induced spin structures is based on the following principle:
%If $\xi=\xi'\oplus \xi''$ are bundles over a CW space $X$, then
%spin structures on any two determine a spin structure of
%the third (cf. \cite[p. 33]{Ki1}).

%Suppose $\imath:M\hookrightarrow\RR^{p+2}$ is a connected, closed, oriented $p$-dimensional 
%smooth submanifold of $\RR^{p+2}$, $p\geq 1$. Since any closed oriented
%smooth submanifold of $\RR^{p+2}$ has trivial Euler class, the normal bundle of $M$ in
%$\RR^{p+2}$ is trivial as $M$ is codimension $2$. On the other hand,
%it is well-known that there exists
%a \emph{Seifert hypersurface} $\Sigma\subset\RR^{p+2}$ of $\imath(M)$, 
%namely, a compact connected oriented $(p+1)$-dimensional smooth submanifold
%such that $\partial\Sigma=\imath(M)$.

The trivialization
$(W,H)$ of $N(F)$ defines a spin structure $\sigma'$ of $N(F)$,
and the canonical spin structure $\varsigma^{4}$ of $\RR^{4}$ restricts to a spin structure
on $T\RR^{4}|_{F}$. As $T\RR^{4}|_{F}=TF\oplus N(F)$,
we conclude there is a spin structure $\sigma$ of $TF$ such that $\sigma\oplus\sigma'=\varsigma^{4}$ by (ii).

For any loop $\alpha$ on $F$, when pushed into the interior of $\Sigma$ along $W$, becomes null-homologous in $\RR^4\setminus F$.  
 Then one can verify that the spin structure $\sigma$ on $F$ is independent of the choice of $\Sigma$ and $(W,H)$. 
Hence for the embedding $F \hookrightarrow\RR^{4}$, we define the \emph{induced spin structure} as:
$$e^*(\varsigma^{4})=\sigma,$$
where $\sigma$ is as described above.  Moreover one can check that $\sigma$  bounds the  spin structure on $\Sigma$ induced from
$H$ and $\varsigma^{4}$. 

\begin{theorem}\label{spin-obstruction1} \cite{DLWY}
For each smooth embedding $e: F_g\to \RR^{4}$, 

(1) there is an induced spin structure $\sigma=e^*(\varsigma^{4})$ on $F_g$ from the
embedding, which is bounded;

(2) for any $f\in Aut(F_g)$, if  $f$ is topologically extendable over $\RR^{4}$ for $e$, then $f^*(\sigma)=\sigma$.
\end{theorem}

\section{An extending criterion}

\begin{theorem}\label{criterion} 
Suppose $f\in Aut(F_g)$. Then the following are equivalent:

(1) There is a bounded  $f$-invariant spin structure $\sigma$ on $F_g$.

(2) $f$ is topologically extendable over $S^4$ for  some smooth embedding  $e: F_g\to S^{4}$,

(3) $f$ is extendable over $S^4$ for  some trivial embedding  $e: F_g\to S^{4}$.
 
 \end{theorem} 

Let $q_\sigma$ be the corresponding quadratic form for the induced spin structure $e^*(\varsigma^4)=\sigma$. 
Lemma \ref{Rohlin} below had appeared in an early version  of \cite{DLWY}, and might be known  for experts.

\begin{lemma}\label{Rohlin} Suppose $e:  F_g \hookrightarrow\RR^{4}$ is a smooth embedding. 
The induced spin structure $q_\sigma $ coincides with the Rohlin form $q_e$.\end{lemma}

\begin{proof} To verify this, consider a smoothly embedded surface $P\subset\RR^4$ as in the definition of $q_e$. Note $\partial P$ becomes null-homologous in $\RR^4\setminus F$ if we push $\partial P$ along the normal vector field of $\partial P$ in $P$,  hence the normal vector field of $\partial P$ in $P$ is equivalent to $W|_{\partial P}$, where $W$ is the vector field  in definition of induced spin structure. 

There is a trivialization defined by a frame field $(U,V,W,H)|_{\partial P}$ such that for any $x\in \partial P$, $U_x\in T(\partial P)|_x$, $V_x\in N_{F_g}(\partial P)|_x$, $W_x\in N_{P}(\partial P)|_x$, and $H_x$ is the restriction of  the vector field $H$ in the definition of induced spin structure.

%a complementary vector orthogonal to $U_x,V_x,W_x$. 
%Then $(W,H)|_{\partial P}$ here is equivalent to the trivialization $(W,H)$ of $N(F)$  (in the definition of induced spin structure) restricted on $\partial P$.

Now $(U,V,W,H)\simeq\varsigma^4|_{\partial P}$ if and only if it extends over $T\RR^4|_P$. 

Suppose $(U,V,W,H)\simeq\varsigma^4|_{\partial P}$. Then by definition, the restriction of the induced spin structure $\sigma=e^*(\varsigma^4)$ on ${\partial P}$ is  given by $(U,V)|_{\partial P}$. Since  $(U,V)|_{\partial P}$ is the Lie-group spin structure, $q_\sigma(\partial P)=1$.

Since $(U,W)|_{\partial P}$ is the Lie-group spin structure, it does not (stably) extend over $TP$.
Then  $(U,V,W,H)\simeq\varsigma^4|_{\partial P}$ if and only if $(V,H)|_{\partial P}$ fails to extend (stably) over $N_{\RR^4}(P)$, i.e. $|P\cap P'|$ is odd. That is to say $q_e([\partial P])=1$.

We conclude $q_\sigma([\partial P])=1$ if and only if 
$q_e([\partial P])=1$.
\end{proof}

By Lemma \ref{Rohlin},  we restate Theorem \ref{spin-obstruction2} in term of induced spin structure.

\begin{theorem}\label{spin-obstruction3} (\cite{Hi}, cf. also \cite{Mon} for $g=1$)
For a trivial embedding  $e: F_g\to S^{4}$ and $f\in Aut(F_g)$,  $f$ is extendable over $\RR^4$ if and only if $f$ preserves the induced spin structure.
\end{theorem} 

\begin{proof}[Proof of Theorem \ref{criterion}]
Suppose $f:  F_g\to F_g$ is topologically extendable over $S^4$ for some smooth embedding $e: F_g\to S^4$. 
By classical fact in manifold topology, we always can assume that the extension $\tilde f: S^4\to S^4$ has a fixed point $x$
in $S^4\setminus e(F_g)$,  therefore $f$ is topologically extendable over $\RR^4$ for the embedding $e: F_g\to \RR^4=S^4\setminus x$. 
Hence there must be
a bounded spin structure on $F_g$, which is   $f$-invariant by Theorem \ref{spin-obstruction1}. So (2) implies (1).

Suppose  $\sigma_1$ is a bounded spin structure which is $f$-invariant, that is 
$$f^*(\sigma_1)=\sigma_1.$$

Let $$e: F_g\to \RR^4 =S^4\setminus x\subset S^4$$ be any trivial embedding, that is $F=e(F_g)$ bounds a 3-dimensional handlebody in $S^4$.
This provided a bounded spin structure
$\sigma=e^*(\varsigma^{4})$ on $F_g$ by Theorem \ref{spin-obstruction1}.

Since the mapping class group acts transitively on the bounded spin structures by Lemma \ref{trans}, there is a diffeomorphism $\phi: F_g\to F_g$
such that $$\phi^*(\sigma)=\sigma_1.$$ 
Then 
$$e\circ \phi : F_g\to S^4$$
is another trivial embedding with the image $F$, and 
$$(e\circ \phi)^*(\varsigma^{4})=(\phi^*\circ e^*)(\varsigma^{4})=\phi^*( e^*(\varsigma^{4}))=\phi^*( \sigma)=\sigma_1.$$
That is to say, the induced  bounded spin structure of the trivial embedding $e\circ \phi : F_g\to S^4$ is the $\sigma_1$, which is $f$
invariant. Then by Theorem    \ref{spin-obstruction3}, $f$ is extendable as diffeomorphism for the trivial embedding  $e\circ \phi : F_g\to S^4$.  So (1) implies (3).

Clearly (3) implies (2).

So we prove the theorem.
%Since the extension made by Hirose in \cite{Hi} has compact support, $f$ extendable as diffeomeophism for the smooth embedding  $e\circ \phi : F_g\to \RR^4=S^4\setminus x\subset S^4$.
\end{proof}

\begin{remark} 
(1)   Theorem \ref{criterion} is a reduction, which transfers the problem to extend $f\in Aut(F_g)$ over $S^4$ for  some embedding  $F_g\to S^{4}$ to  the problem to find  a fixed point of $f$ on $\mathcal B_g$. %This is a reduction, since
% the structure of $\mathcal B_g$ is studied.

(2) By  Hurwitz and Nielsen's theories,  each periodic map $f$ on $F_g$ is corresponding to  a set of numerical data.
However in practice, to see if $f$ has a fixed point  on  
$\mathcal B_g$,  we may need either (i) to find  a geometric model for the action of $f$ on $F_g$, as in the proof of Theorem \ref{4g+2,4g},
or (ii) to apply number theory when  $f$ has  some special period,  as in the proof of Theorem \ref{3-5-7};
or (iii) to present $f$ as a product of certain standard  generators of $\mathcal M(F_g)$, as carried in \cite{Mon} for $g=1$, the torus. 
Both (i) and (iii) are not easy in general.
 \end{remark}

%\begin{remark} 
%In \cite{Hi} the target is the sphere, and 
%in \cite{DLWY} the target is the Euclidean space.
%For those two different targets, there is no difference for our extendable problem. 
%(But may be a difference for the periodical extendable problem).
%To unify the terminology, we state both Theorem \ref{spin-obstruction1} and \ref{spin-obstruction2} for spheres,
%where the spin structure on $F_g$ induced from  $e: F_g\to S^{4}$
%mean the  spin structure induced from  $e: F_g\to \RR^{4}=S^{4}\setminus \text{a point}$,
%\end{remark}

\section{Extending periodic maps as  group actions}
The fact below is used in the proofs of Theorem \ref{group-ext} and Theorem \ref{4g+2,4g}.
\begin{lemma}\label{Wiman}
Suppose $w$ and $u$ are two periodic maps of order $k$ on $F_g$, $k=4g+2$ or $k=4g$. Then $w$ is (periodically) extendable if and only if $u$ is (periodically) extendable.
\end{lemma}

\begin{proof}
We need only to argue that $u_g$ is (periodically) extendable implies that $w_g$ is (periodically) extendable.
Let $\left<w \right>$ and $\left<  u \right>$ be the group generated by $w$ and $u$ respectively.
Then both of them are cyclic groups of order $k$.
Since $k=4g$ or $4g+2$,  it is known that those two groups are conjugated \cite{Ku},
that is to say, there is an automorphism $h$ on $F_g$ such that 
$$h^{-1}\left<w\right > h=\left<u\right >.$$
So $h^{-1}\circ w\circ h=u^p$ for some $p$ which co-prime with $k$.
 
Suppose $u$ is (periodically) extendable for some embedding $e: F_g\to \RR^4$. Then there exists $\tilde u\in Aut(\RR^4)$ 
($\tilde u$ is an element $\in Aut(\RR^4)$ of order $k$)
such that
$$  \tilde u \circ e = e\circ u.$$
Then clearly 
$$  \tilde u^p \circ e = \tilde u^{p-1}\circ  (\tilde u\circ e) =\tilde u^{p-1} \circ( e\circ u)= ...=e\circ u^p.$$
That is to say $u^p$ is (periodically) extendable for $e$. Then from 
$$\tilde u^p \circ e =e\circ u^p=e\circ h^{-1}\circ w\circ h,$$
we have 
$$\tilde u^p \circ e\circ h^{-1} =e\circ h^{-1}\circ w.$$
Let $\tilde w=\tilde u^p$ and $e'=e\circ h^{-1}:F_g\to \RR^4$.
Then we have 
$$\tilde w \circ e' =e'\circ w.$$
That is $w$ is (periodically) extendable w.r.t. $e'$.
\end{proof} 

\begin{proof}[{Proof of Theorem \ref{group-ext}}]

(1)  Now we consider $\RR^4$ as $4$-dimensional Euclidean space, and $S^3$ as its unit sphere.
View $\RR^{4}$ as the product  
$$\RR^{4}=\CC_1\times \CC_2.$$
where each $\CC_j$ is a complex line  with a complex coordinate $z_j$.
Let $S^1_j$ be the unit circle of $\mathbb{C}_j$, $j=1,2$.
Then
  
$$S^3=\{(z_1,z_2)\in\mathbb{C}^2\mid |z_1|^2+|z_2|^2=1\}.$$
The Clifford torus given by 
$T=\{z_1, z_2\in S^3||z_1|=|z_2|\}$ divides $S^3$ into two solid tori given by 
$$H_1= \{z_1, z_2\in S^3||z_1|\ge|z_2|\},\,\,H_2= \{z_1, z_2\in S^3||z_1|\le|z_2|\}.$$
 $H_1$ and $H_2$  have $S^1_1\times \{0\}$ and $\{0\}\times S^1_2$ as the center circle respectively. 
For short, below we still denote $S^1_1\times \{0\}$ by $S^1_1$ and $\{0\}\times S^1_2$ by $S^1_2$.

Define a cyclic group action $\tau_{m,n}\in SO(4)$ on the triple $(\RR^4, S^3, T)$ by:
$$\tau_{m,n}: (z_1,z_2)\mapsto(e^{\frac{2\pi i}{m}}z_1,e^{\frac{2\pi i}{n}}z_2).$$
The order of $\tau_{m,n}$ is the least common multiple of $m$ and $n$.

\begin{figure}[htbp]
\begin{center}
\includegraphics[width=240pt, height=180pt]{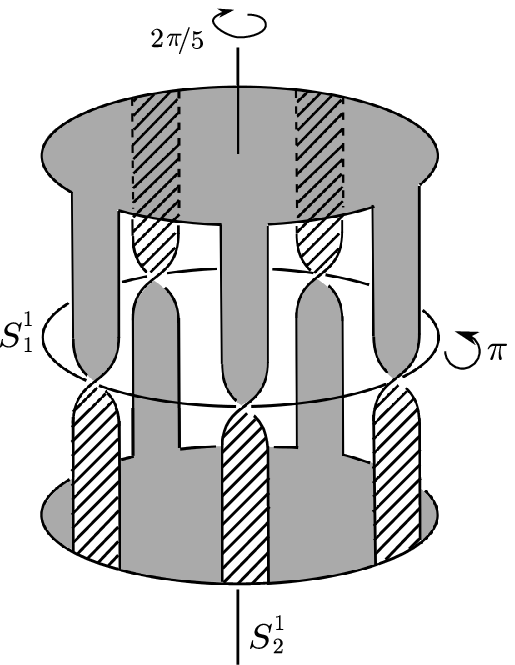}
\end{center}
\centerline{Figure 1}
\end{figure}

In Figure 1 (where $g=2$), the bounded surface $X$ is obtained from two round discs evenly connected by $2g+1$ half twisted bands, and with
a given embedding
$$X\subset \RR^3\subset S^3\subset \RR^4$$ 
where $S^1_1$ is the circle crossing the $2g+1$ bands at their middle, and $S_2^1$ the vertical line crossing the centers of two discs
adding a point at infinity.

Clearly $X$ is invariant under the $\pi$-rotation around the axis $S^1_1$ and under the $\frac {2\pi}{2g+1}$-rotation around the axis
$S^1_2$. That is to say 
we have periodic map 
$$\tau_{2,2g+1}: (S^3, X)\to (S^3, X)$$
of order $4g+2$.

It is easy to see that $\partial X$, the boundary of $X$, is  connected, and $X$ is oriented. 
Since $X$ has deformation retractor which is a graph with two vertices and $2g+1$ edges,
the Euler Characteristic 
 $$\chi(X)=2-(2g+1) =1-2g.$$
 We conclude that $X$ is obtained from $F_g$ with one disc $D$ removed.
 Clearly we can extend the restriction of $\tau_{2,2g+1}$ on $X$ to the disc $D$ by Alexander trick to obtained a Wiman map $u_g$ on $F_g$.
 
The above discussion tell us that the restriction of $u_g$ on $X$  is 
periodically extendable over $S^3$  for the embedding $X\subset S^3$ given by Figure 1.
Now we are going to extend our embedding from $X$ to  the whole $F_g$ into $\RR^4$ so that  $u_g$ is periodically
extendable over $\RR^4$. 
Let $O$ be the origin of $\RR^4$ and $Ox$ be the line segment from
$O$ to $x$ for any point $x\in \RR^4$, then 
$$O*\partial X=\bigcup_{x\in \partial X} {Ox}\subset \RR^4$$
is a disc invariant under the action $\tau_{2,2g+1}$.
Then $$\hat X=X\cup O*\partial X\subset \RR^4$$
is a topological embedding of $F_g$ into $\RR^4$ which is invariant under the  action $\tau_{2,2g+1}$,
and the restriction of $\tau_{2,2g+1}$ on  $F_g=\hat X$ is the Wiman map $u_g$. 

Consider $S^4$ as the one point compactification of $\RR^4$. Since $\tau_{2,2g+1}\in SO(4)$, then $\tau_{2,2g+1}$ can be extended to a diffeomorphism on $S^4$, denoted by  $\tilde u_g$, we have the embedding
$$e: F_g=\hat X \subset  R^4\subset S^4$$
and 
$$\tilde u_g \circ e= u_g \circ e.$$
So the Wiman map $u_g$ is periodically extendable. Then by Lemma \ref{Wiman}, we prove (1).

\begin{remark} (i) The construction of the embedding $e: F_g\subset \RR^4$ 
is based on an observation that even a Wiman map $w_g$ on $F_g$ is not periodically  extendable over $S^3$, but it is periodically 
extendable over $S^3$ after making a suitable puncture, that is to remove a $w_g$-invariant disc.
Recall the fact that no closed lens spaces are embeddable into $S^4$ \cite{Han}, but some of them are embeddable into $S^4$ after
punctured \cite{Zee}.
    
(ii) The embedding $e$ is not smooth,  indeed is not locally flat at $O$. %\cite[page 84]{Rol}.

(iii) In general $X$ in Figure 1 is the Seifert surface of torus knot of type $(2, 2g+1)$. This provides another way to see the genus and symmetry of $X$.
\end{remark}
 
(2) We need only to prove (2) for a Wiman map $u_g$ described in (1) by Lemma \ref{Wiman}. Then we need only to prove that $f=u_g^2$ is not periodically extendable
for any smooth embedding.

For each map $f:X\to X$, denote the fixed point set of $f$ by $\FF(f)$.

Suppose $F_g\subset S^4$ and there is a smooth periodic map $\tilde f$ on $S^4$ such that the restriction of $\tilde f$ on $F_g$ is $f$.
To reach a contradiction, we need a sequence of known facts and observations: 

(i)  $f: F_g\to F_g$ has exactly three fixed points,
which  is an observation from Figure 1, the  two fixed points of $f$ on $F_g=\hat X$ are the  centers of two round disks in Figure 1 and another fixed point is the center of the disk $D$ attached to $X$.  

(ii) Since $f$ is of order $2g+1$,  $\tilde f$ is of order $2g+1$. Since $\tilde f$ is of odd order, $\tilde f$ must be orientation preserving.

(iii)  Since $\tilde f$ is a periodic map on $S^4$,  according to the classical 
Smith theory \cite{Sm},
$\FF(\tilde f)$ must be a homology $k$-sphere $H^k\subset S^4$, $k=0,1,2,3,4$. The case $k=4$ is ruled out since $\tilde f$ is
not the identity. By (ii) $\tilde f$ is orientation preserving. Then cases $k=1,3$ are ruled out by Lemma \ref{even} below.  Since $0$-homology sphere is the $S^0$,
which consists of two points, and  $\FF(\tilde f)$ contains at least three points by (i),  the case $k=0$ is ruled out. We conclude $k=2$.

(iv) Since $\tilde f$ is a smooth periodic map on $S^4$, a classical fact in differential topology is that $\FF (\tilde f)$ is a smooth sub-manifold of $S^4$ \cite[page 74]{Wal}.
By the conclusion of (iii),  and the fact that a 2-manifold with homology as the 2-sphere is the 2-sphere,
we obtained that $\FF (\tilde f)$  is a smooth 2-sphere $S$ in $S^4$.

 \begin{lemma}\label{even}
Suppose $f$ is an orientation preserving smooth periodic map on $S^m$. Then the co-dimension of $\FF(f)$ in $S^m$ is even.   
 \end{lemma}
 
 \begin{proof}  
 Suppose first $A: \RR^m\to \RR^m$ is a linear periodic map. 
 Then $A\in GL_m(\RR)$ has a Jordan canonical form $J$ in the orthogonal group $O(n)$, that is to say there is $P\in GL_m(R)$
 such that $A=P^{-1}JP$. We may assume that $P=I_m$, then 
 $A=J$.
  Correspondingly, there is a decomposition 
 $$\RR^m=V_{-1}\oplus V_0 \oplus V_{1}\oplus ....\oplus V_{k},$$
where $V_j$ are $A$-invariant subspace,  $\text{dim}V_j\ge 0$ for $j=-1,0$, and $\text{dim}V_{j}=2$ for $j=1,..., k$, such that 
for each $$x=x_{-1}+ x_0+ x_1+...+ x_k,$$ where $x_j\in V_j$, we have

$$Ax= -x_{-1} +x_0+r_1(x_1)+...+ r_k(x_k) \qquad (3.1)$$
where each $r_j\in SO(2)$ is a periodic map on $V_j$ of order $>2$ for $j>0$.

According to (3.1), $\text{dim}\FF(A)=\text{dim} V_0$. So co-dimension of $\FF(A)$ in $\RR^m$ is $\text{dim} V_{-1}+2k$. 
 Since each $r_j$ 
 is in $SO(2)$, the determinant of $A$ is $(-1)^{\text{dim} V_{-1}}$. 
  If $A$ is orientation preserving, we must have ${\text{dim} V_{-1}}$ is even.  
 So co-dimension of $\FF(A)$ in $\RR^m$ is even. 
 
Now we consider the smooth periodic map $f$  on $S^m$. A classical result claims $\FF(f)$ in a smooth sub-manifold of $M$ \cite{Wal}.
Pick any point $x\in \FF(f)$.  Since $f$ is orientation preserving, the tangent map $df_x : T_xM\to  T_xM$ is orientation preserving linear map. 
 By the conclusion we just get in last paragraph,  the co-dimension of $\FF(df_x)$ in $T_xM=\RR^m$ is even. Since $\text{dim} \FF(f)=\text{dim} \FF(df_x)$, the lemma follows.
\end{proof}

\begin{lemma}\label{transverse}
The two smooth manifolds $F_g$ and $S$ above meet transversely in $S^4$.
\end{lemma}
\begin{proof}
For any $x\in F_g\cap S$, we have 
$$\tilde f (x)=x,\,\,  \tilde f (F_g)=F_g,\,\,   \tilde f(S)=S,$$ 
and $\tilde f|S$ is the  identity map on $S$.

Let $df_x: T_x S^4\to T_x S^4$ be the tangent map of $f$ at $x$, then we have 

$$d\tilde f_x|T_x F_g=T_x F_g,\,\,  d\tilde f_x| T_x S=T_x S,$$

Note that $d\tilde f_x |T_x F_g$ does not have real eigenvalue, and $ d\tilde f_x| T_x S$ is the identity map on $T_x S$.
Since $d\tilde f_x |T_x F_g$ and  $ d\tilde f_x| T_x S$ have no common eigenvalues, we have 
$$T_x F_g\cap  T_x S=0\in T_x S^4,$$
then we have 

$$T_x S^4 =T_x F_g\oplus  T_x S.$$
That is to say $F_g$ and $S$ meet transversely in $S^4$.
\end{proof}

We reach a contradiction as below: On one hand, since the embedding $F_g$ is homotopic to a constant map in $S^4$, the
$\text{mod}$ 2 algebraic intersection number of $F_g$ and $S$ must be 0 \cite{Wal}. On the other hand,
by Lemma \ref{transverse} and Observation (i), $F_g$ and $S$ meet transversely in three points, therefore  their  $\text{mod}$ 2 algebraic intersection number
must be 1. This is a contradiction.
\end{proof}

\section{Extending periodic maps of  large orders}

It is well-known fact that $F_g$ can be obtained by identifying the opposite edges of regular $4g$-gon $\Gamma_{4g}$,
see Figure 2 for $g=2$. The oriented edges presented by the same letter are identified. All corner points are identified to a point.
\begin{figure}[htbp]
\begin{center}
\includegraphics[width=130pt, height=130pt]{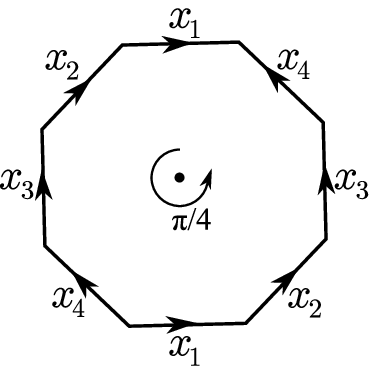}
\end{center}
\centerline{Figure 2}
\end{figure}

Let $v_g$ be a map of order $4g$ on $\Gamma_g$ which is a rotation of angle $\pi/2g$
around the center. It is easy to see that $v_g$ induces a map of order $4g$ on $F_g$, still denoted as $v_g$.

We can also think of the $4g$-gon $\Gamma_{4g}$ of $F_g$ as the union of two regular $(2g+ 1)$-gon’s which lies in the plane, see 
Figure  3 for $g = 2$.

\begin{figure}[htbp]
\begin{center}
\includegraphics[width=230pt, height=130pt]{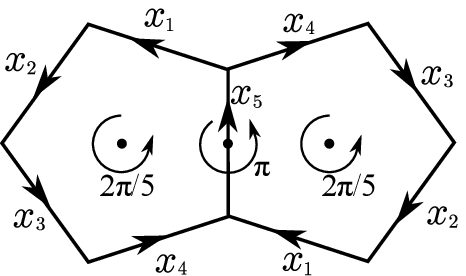}
\end{center}
\centerline{Figure 3}
\end{figure}

Let $\tau_g$ be a map of order $2g+1$ whose restriction on each regular $(2g+1)$-gon is a rotation of angle $2\pi/(2g + 1)$
around the center. Let $\eta$ be a rotation of angle $\pi$ in the plane which switches the two regular $(2g+1)$-gon’s. 
The composition of these two maps 
$w_g=\eta\circ \tau_g^{g+1}$
is a map of order $4g +2$.
Then it is easy to see  each one of the above three maps induces a map on $F_g$ of the same order, still denoted by $\tau_g$ and $\eta$
and  $w_g=\eta\circ \tau_g^{g+1}.$ Then $w_g$ is a Wiman map on $F_g$. Since  $\eta$ and $\tau_g$ commute,
we have $$w_g^2=\eta\circ \tau_g^{g+1}\circ \eta\circ \tau_g^{g+1}=\eta^2\circ  \tau_g^{2(g+1)}=\tau_g,$$
since $\eta^2$ and $\tau_g^{2g+1}$ are the identity on $F_g$.

%The image of the corner points is a fixed point and the image of the centers of two $(2g+1)$-gon’s is a periodic orbit of length 2.

Convention: In this section, we will often use the same letter $x$ to present a closed path on $F_g$ and its homology class  in
$H_1(F_g, \ZZ)$, and in $H_1(F_g, \ZZ_2)$.
Similarly  we often use the same letter $f$ to present an diffeomorphism on $F_g$ and its induced map on $H_1(F_g, \ZZ_2)$. 
%Below we denote  
%$$v_g=\tau_g,\ \ f_2=\phi_g.$$  

When we identify  the opposite edges of regular $4g$-gon $\Gamma_{4g}$ to get $F_g$, the quotient of each $x_i$ is a
simple closed curve, still denoted by $x_i$, in $F_g$. 
We have the following elementary geometric facts:

(1) The homology classes $x_1, x_2,... x_{2g-1}, x_{2g}$ is a basis of $H_1(F_g, \ZZ)$, hence in $H_1(F_g, \ZZ_2)$, and we have
(see Figure 3)
$$x_1+ x_2+... +x_{2g-1}+x_{2g}+x_{2g+1}=0 \qquad (G1)$$
in $H_1(F_g, \ZZ)$, hence in $H_1(F_g, \ZZ_2)$.

(2)  The action of $\tau_g$ forms a closed orbit
$$x_1\to x_2\to x_3\to ... \to  x_{2g}\to  x_{2g+1}\to x_1 \qquad (O1)$$
on both $H_1(F_g, Z)$ and $H_1(F_g, Z_2)$. The action of $v_g$ forms a closed orbit
$$x_1\to x_2\to x_3\to ... \to  x_{2g}\to -x_1\to -x_2\to -x_3\to ... \to -x_{2g} \to x_1 $$
on $H_1(F_g, Z)$, and a closed orbit
$$x_1\to x_2\to x_3\to ... \to  x_{2g}  \to x_1, \qquad (O2)$$
on $H_1(F_g, Z_2)$.

\begin{lemma}\label{GI} The $\ZZ_2$-intersection number 
$$x_i\cdot x_j=1\,\,  \text {for  $i\neq j$}\qquad (G3).$$
\end{lemma}

\begin{proof} %There are many way to see this fact. % $x_i$ and $x_j$  meet transversely, therefore the lemma is proved.
%When we identify  the opposite edges of regular $4g$-gon $\Gamma_{4g}$ to get $F_g$, the quotient of each $x_i$ is
%simple closed curve, still denote by $x_i$, in $F_g$, and
%The geometric intersection number of simple closed curves $x_i$ and $x_j$ on $F_g$ is 1 for  $i\neq j$.
By Figures 2 and 3, one can observe that for  $i\neq j$, two simple closed curves $x_i$ and $x_j$ in $F_g$ meet transversely
exactly one point $O$, the quotient of the corner of    $\Gamma_{4g}$. Then the lemma is proved.
%There are various way to see $x_i$ and $x_j$ indeed meet transversely, 

Another proof is to  think $\Gamma_{4g}$ as a hyperbolic regular $4g$-gon with angle $\frac {2\pi}{4g}$
at each corner. Now the quotient $F_g$ is a hyperbolic surface. Then  for $i\neq j$ and $i, j\in \{1,..., 2g\}$, 
$x_i$ and $x_j$ are geodesics meet at $O$ with angle 
$\frac {2|i-j|\pi}{4g}$, so   $x_i\cdot x_j=1$.
For $i\ne 2g+1$, by (G1), $x_i\cdot x_i=0$ and the fact we just proved, we have 
 $$x_i\cdot x_{2g+1}=x_i\cdot (x_1+...+x_i+...+x_{2g})=2g-1\equiv 1 \,\, \text{mod}\,\, 2.$$
 The lemma follows.
%Another way is directly to see Figure 3 how those $x_i$'s stay in $F_g$, where $g=3$.
%\begin{figure}[htbp]
%\begin{center}
%\includegraphics[width=260pt, height=140pt]{3.eps}
%\end{center}
%\centerline{Figure 4}
%\end{figure}
\end{proof}

From now on we  call the basis $x_1, x_2,..., x_{2g-1}, x_{2g}$  together with $x_{2g+1}$ the  $\tau_g$-orbit,
and call $x_1, x_2,..., x_{2g-1}, x_{2g}$ the $v_g$-orbit.

In term of $q_\sigma$, the formula $(S2)$ becomes 

$$f^*(q_\sigma)(x)=q_\sigma(f(x)) \qquad (S2')$$
for any $q_\sigma \in \mathcal{S}(F_g)$ and $x\in H_1(F_g, \ZZ_2)$.
Suppose $\sigma$ is $f$-invariant, that is $f^*(\sigma)=\sigma$, equivalently $f^*(q_\sigma)=q_\sigma$, then we have

$$q_\sigma(x)=q_\sigma(f(x))\qquad (G4).$$ 
(G4) is the mean of that $\sigma$ is $f$-invariant in term of $q_\sigma$.

%\begin{proposition}\label{spin1}
%Suppose $\sigma$ is a spin structure on $F_g$, and $f^*(\sigma)=\sigma$
%Then $\sigma$ is bounded for bounded for $g=4k$ and $4k+3$, and  unbounded for $g=4k+1$ and $4k+2$.
%\end{proposition}

%Below we often use $q\equiv 0$  on the $f$-orbit to indicate $f$ is the constant 0 on the $f$-orbit, and so on.
\begin{proposition}\label{spin1}
(1)  There exists  $q\in \mathcal B(F_g)$ such that  $\tau_g^*(q)=q$  if and only if  either $g=4k$ or $g=4k+3$.

(2) There exists  $q\in \mathcal B(F_g)$ such that $v_g^*(q)=q$ if and only if either  $g=4k$, or $g=4k+1$, or $g=4k+3$.
\end{proposition}

The proof of Proposition \ref{spin1} follows from a sequence of lemmas below.

Let $I_{m}=\{(i,j)|i,j\in \{1,2,..., m-1,m\}, i< j\}$.

By (S1) and induction we have 
\begin{lemma}\label{addition} 
Let $y_1,..., y_m$ be $m$ elements in $H_1(F_g, \ZZ_2)$. Then
$$q(\sum_{i=1}^m y_i)=\sum _{i=1}^m q(y_i)+\sum _{(i,j)\in I_m} y_i\cdot y_j$$
\end{lemma}

\begin{lemma}\label{01} 
(1) There are exactly  one $\tau_g$-invariant spin structure for each $g$ determined by $q=0$ on the $\tau_g$-orbit for even $g$,
and $q=1$  on the $\tau_g$-orbit for odd $g$.

(2) There are exactly two $v_g$-invariant spin structures for each $g$ determined by either $q=0$ on the $v_g$-orbit or $q=1$ on the $v_g$-orbit.
\end{lemma}

\begin{proof} Suppose $q\in \mathcal S(F_g)$.  Let $f$ be either $\tau_g$ or $v_g$. We first prove that 
$q$ is $f$-invariant if and only if $q$  is the constant with value 0 or 1 on the $f$-orbit. 

Suppose $q$ is $f$-invariant.
 By  (G4),  on the $f$-orbit,  $q$ is a constant, which is either 0 or 1.

Now suppose $q$ is a constant on the $f$-orbit. For any $x\in H_1(F_g, \ZZ_2)$, $x=\sum_{i=1}^{2g}a_ix_i$, where $a_i\in \ZZ_2$,  
and by (O1) or (O2), we have
$$f(x)=f(\sum_{i=1}^{2g}a_ix_i)=\sum_{i=1}^{2g}a_if(x_i)=\sum_{i=1}^{2g}a_ix_{i+1}.$$

Hence

$$f^*(q)(x)=q(f(x))=q(\sum_1^{2g}a_ix_{i+1})$$
$$=\sum_{i=1}^{2g}a_iq(x_{i+1})+\sum _{(i,j)\in I_{2g} } a_ia_jx_{i+1} \cdot x_{j+1}$$.
$$=\sum_{i=1}^{2g}a_iq(x_{i})+\sum_{(i,j)\in I_{2g} } a_ia_jx_{i} \cdot x_{j}=q(\sum_{i=1}^{2g}a_ix_i)=q(x).$$

The fourth equality holds since $q$ is a constant on the $f$-orbit and $x_{i} \cdot x_{j}=1$ for $i< j$ by Lemma \ref{GI},
the third and fifth equalities follow from Lemma \ref{addition}.
So we have $f^*(q)=q$.

For (1),  we have 
$$0=q(x_1+ x_2+... +x_{2g-1}+ x_{2g}+x_{2g+1})$$
$$=\sum_{i=1}^{2g+1} q(x_i)+\sum _{(i,j)\in I_{2g+1} } x_i \cdot x_j$$
$$=(2g+1)q(x_1)+(2g+1)g=q(x_1)+g \,\, \text{mod}  \,\, 2.$$

We obtained the first equality by (G1), the second by Lemma  \ref{addition}, the third by that $q$ is the constant on $\tau_g$-orbit
and Lemma \ref{GI}. So $q(x_1)$ and $g$ have the same parity. 

 For (2), it is obvious since the $v_g$-orbit is a basis of $H_1(F_g, \ZZ_2)$.
\end{proof}

Below $k\equiv t$ means that $k\equiv t\,(\text{mod}\,4)$, and $C_n^k=\frac{n!}{k!(n-k)!}$. Define 
$$\sum_{k\equiv 0}C_{2g}^k=A_0, \,\, \sum_{k\equiv 1}C_{2g}^k=A_1,\,\, \sum_{k\equiv 2}C_{2g}^k=A_2,\,\,\sum_{k\equiv 3}C_{2g}^k=A_3,$$

Recall that $H_1(F_g, \ZZ_2)$ has $2^{2g}$ elements,  and moreover 
$$2^{2g}=\Sigma_{i=0}^{2g} C_{2g}^i =A_0+A_1+A_2+A_3\qquad (1.1).$$

\begin{lemma}\label{0-1}  Let $f$ be either $\tau_g$ or $v_g$. For each $q\in \mathcal S(F_g)$,

(1) if $q=0$ on the $f$-orbit, then $q(x)=0$ for $A_0+A_1$ elements;

(2) if $q=1$ on the $f$-orbit, then $q(x)=0$ for $A_0+A_3$ elements.
\end{lemma}

\begin{proof}

Note for each $x\in H_1(F_g, \ZZ_2)$, there is a unique subset $\{x_{i_1},...,x_{i_m}\}$ of the basis $\{x_1, x_2,..., x_{2g}\}$ of $H_1(F_g, \ZZ_2)$, such that $x=x_{i_1}+...+x_{i_m}$,
and we call such $x$ is of length $m$.
%Each $m$-tuple $\{x_{i_1},...,x_{i_m}\}$ determines an element of $x=x_{i_1}+...+x_{i_m}\in H_1(F_g, \ZZ_2)$,
%$and the converse is also true.

Suppose $x=x_{i_1}+...+x_{i_m}$, according to Lemma \ref{addition} and (G3) we have 

$$q(x_{i_1}+...+x_{i_m})=C_m^2 \qquad (1.3)$$
if $q$ is constant 0 on the $f$-orbit; and 
$$q(x_{i_1}+...+x_{i_m})=C_m^2+m \qquad (1.4) $$
if $q$ is constant 1 on the $f$-orbit. 

Next we verify (1) by using (1.3). Fix $m\in \{0,1,..., 2g\}$, we have 

$$C_m^2=\frac{m(m-1)}2=\left\{
\begin{array}{ccc}
\text{even}, \,\, & \text{for} \,\,m=4l, 4l+1;\\
\text{odd}, \,\, & \text{for} \,\,m=4l+2, 4l+3.
\end{array}
\right. $$ 

By (1.3), we obtained  that $q(x)=0$ on those elements $x$ of length $4l$ and of length $4l+1$; 
and  $q(x)=1$ on those elements $x$ of length $4l+2$ and  of  length $4l+3$. 

Recall for each $m\in \{0,1,..., 2g\}$,  
there are $C_{2g}^m$ elements of length $m$, and  $\sum_{k\equiv i}C_{2g}^k$ is used to define $A_i$. So when $m$ runs over $\{0,1,...,2g\}$, 
$q(x)=0$ on  $A_0+A_1$ elements.
This is the conclusion of (1).

Next  we verify (2) by using (1.4). Fix $m\in \{0,1,..., 2g\}$,  we have  

$$C_m^2+m=\frac{m(m-1)}2+m =\left\{
\begin{array}{ccc}
\text{even}, \,\, & \text{for} \,\,m=4l, 4l+3;\\
\text{odd}, \,\, & \text{for} \,\,m=4l+1, 4l+2.
\end{array}
\right. $$ 

By (1.4), we obtained  that $q(x)=0$ on those elements $x$ of length $4l$ and of length $4l+3$; 
and  $q(x)=1$ on those elements $x$ of length $4l+1$ and of length $4l+2$. 
So when $m$ runs over $\{0,1,...,2g\}$, 
$q(x)=0$ on  $A_0+A_3$ elements.
This is the conclusion of (2).
\end{proof}

\begin{lemma}\label{A_i}
Suppose $g>0$.

(1) $ A_0+A_1=2^{2g-1}+2^{g-1}$ if and only if  $g\equiv 0,1$;

(2) $ A_0+A_3=2^{2g-1}+2^{g-1}$ if and only if $g\equiv 0,3$.
\end{lemma}

\begin{proof} Note first  that 
$$ A_0+iA_1-A_2-iA_3=(1+i)^{2g}.$$

Then

$$A_0-A_2=Re [(1+i)^{2g}]=Re [(\sqrt 2 e^{i\pi/4}) ^{2g}]=2^gRe [(e^{i\pi/4}) ^{2g}]=
\left\{
\begin{array}{ccc}
2^g, \, & g\equiv 0;\\
0, \, & g\equiv 1,3;\\
-2^g, \, & g\equiv 2.
\end{array}
\right.$$

$$A_1-A_3=Im [(1+i)^{2g}]=Im  [(\sqrt 2 e^{i\pi/4}) ^{2g}]=2^gIm [(e^{i\pi/4}) ^{2g}]=\left\{
\begin{array}{ccc}
0, \, & g\equiv 0, 2; \\
2^g, \, & g\equiv 1;\\
-2^g, \, & g\equiv 3.
\end{array}
\right.$$

By adding and subtracting last two identities we have

$$(A_0+A_1)-(A_2+A_3)=\left\{
\begin{array}{ccc}
2^g, \,\, & g\equiv 0, 1; \\
-2^g, \,\, & g\equiv 2, 3.
\end{array}
\right.$$

$$(A_0+A_3)-(A_2+A_1)=\left\{
\begin{array}{ccc}
2^g, \,\, & g\equiv 0, 3; \\
-2^g, \,\, & g\equiv 1, 2.
\end{array}
\right.$$
which implies that $ A_0+A_1=2^{2g-1}+2^{g-1}$ if and only if $g\equiv 0, 1$ and $ A_0+A_3=2^{2g-1}+2^{g-1}$ if and only if $g\equiv 0, 3$.
\end{proof}

\begin{proof}[{Proof of Proposition \ref{spin1}}]  Suppose $q\in \mathcal B(F_g)$.

(1) By Lemma \ref{01} (1),
$q$ is $\tau_g$-invariant if and only if either 
(i) $q=0$  on the $\tau_g$-orbit and $g$ is even, or 
(ii) $q=1$  on the $\tau_g$-orbit and $g$ is odd. 

In case (i), $q=0$ on $A_0+A_1$ elements by Lemma 
\ref{0-1}. Then $q$ is bounded  if and only if $A_0+A_1=2^{2g-1}+2^{g-1}$ by Lemma \ref{BU}. This is the case if and only if  $g=4k$  by Lemma \ref{A_i}.

In case (ii)  $q=0$ on $A_0+A_3$ elements by Lemma 
\ref{0-1}. Then $q$ is bounded if and only if  $A_0+A_3=2^{2g-1}+2^{g-1}$ by Lemma  \ref{BU}. This is the case if and only if $g=4k+3$ by Lemma \ref{A_i}.

(2) By Lemma \ref{01} (2),
$q$ is $v_g$-invariant if and only if either (i) $q=0$ on the $v_g$-orbit or (ii) $q=1$ on the $v_g$-orbit.
As argued in (1), in case (i) we have  $g=4k$ or $4k+1$, and in case (ii) we have $g=4k$ or $4k+3$.
\end{proof}

\begin{lemma}\label{w_g=tau}
$q$ is $\tau_g$ invariant if and only if $q$ is $w_g$ invariant.
\end{lemma}

\begin{proof}
From Figure 2 and the description of $\eta$, we have $\eta(x_i)=-x_i$
for each $x_i$ in $\tau_g$-orbit, and then  $$\eta_*: H_1(F_g, \ZZ_2)\to H_1(F_g, \ZZ_2)$$
is the identity, that is to say,  
$${w_g}_*=\eta_*\circ {\tau_g}_*^{g+1}={\tau_g}_*^{g+1}.$$
Combining ${w_g}_*^2={\tau_g}_*,$ the Lemma follows from the definition.
\end{proof}

\begin{proof}[Proof of Theorem \ref{4g+2,4g}] 
Now Theorem  \ref{4g+2,4g} follows from Proposition \ref{spin1}, Lemma \ref{w_g=tau} and Theorem \ref{criterion}.
\end{proof}

\section{Extending periodic maps of prime orders}
\begin{proposition}\label{B-U}
Let $f: F_g\to F_g$ be a map of order $p^m$, where $p$ is an odd prime and $m$ is an integer.

(i) If p is not a divisor of $2^{2g-1}+2^{g-1}$, then there is a bounded $f$-invariant spin structure.

(ii) If p is not a divisor of $2^{2g-1}-2^{g-1}$, then there is an unbounded $f$-invariant spin structure.
\end{proposition}

\begin{proof}
(i) Let $r=p^m$ and 
$$G=\ZZ_r=\left< f \right>$$
 acts on $\mathcal{B}_g$. For any $g\in  \left< f \right>$ and $q\in \mathcal{B}_g$.  Denote by 
$$Orb(q)=\{g^*(q)| g\in G\}$$
and 
$$Stab_q=\{g\in G| g^*(q)=q\}$$
Since $p$ is prime and $|G|=p^m$, $Stab_q$ is a subgroup of $G$, we have $|Stab_q|=p^u$ for $0\le u \le m.$
Then 
$$|Orb(q)|=\frac {|G|}{|Stab_q|}=p^s$$ for some $0\le s \le m$.

Let  $Orb(q_1),..., Orb(q_l)$ be disjoint orbits such that
$$\mathcal{B}_g=\bigcup_{i=1}^l Orb(q_i)$$
  
Since $|\mathcal{B}_g|=2^{2g-1}+2^{g-1}$ by Lemma \ref{cardinality}, and $2^{2g-1}+2^{g-1}$ is not a multiple of $p$,
we have $ |Orb (q_j)|$ is not multiple of $p$  for some $j\in \{1,..., l\}$.

It follows that $|Orb (q_j)|=1$. That is $f^*(q_j)=q_j$. So $f$ has an invariant bounded spin structure.

(ii) The proof is the similar to that of (i).
\end{proof}

\begin{lemma}\label{many primes} (H. Sun) There are infinitely many prime numbers that are not divisors of any $2^{2g-1}+2^{g-1}$.
\end{lemma}

\begin{proof} We first prove the following

{\bf Claim:} Each prime number in the form $p=8k+7$ is not a factor of $2^g+1$ for any natural number $g$.

 Suppose that $p$ is a divisor of $2^g+1$, then $2^g\equiv -1 \,\text{mod} \, p$ holds.

Since $p$ is in the form of $8k+7$, $2$ is a quadratic residue of $p$ by \cite[Theorem 3.16]{EW}, that is, there exists an integer $k$ such that $$k^2=2 \, \text{mod} \, p.$$

So we have $$-1=2^g=k^{2g} \, \text{mod} \, p.$$ Thus $-1$ is also a quadratic residue of $p$. By \cite[Theorem 3.13]{EW}, such $p$ must be in the form of $4l+1$, but our $p$ is in the form of $4l+3$. So the claim is proved.

By the Dirichlet's theorem \cite[Theorem 10.5]{EW}, there are infinitely many primes in the form of $8k+7$ which are odd. Then the lemma follows from the Claim  and $2^{2g-1}+2^{g-1}=2^{g-1}(2^g+1)$.
\end{proof}

\begin{proof}[Proof of Theorem \ref{3-5-7}]
(i) When  $p=3$ and  $g$ is even, let $g=2k$. Then 
$$2^{2g-1}+2^{g-1} \equiv  (-1)^{2g-1}+(-1)^{g-1} \equiv (-1)^{4k-1}+(-1)^{2k-1} \equiv -2 \,  \text{mod} \, 3.$$
That is to say $3$ is not a divisor of $2^{2g-1}+2^{g-1}$.

(ii) Suppose  $p=5$.   If $g\equiv 0 \,  \text{mod} \, 4$, let $g=4k$. Then 
$$2^{2g-1}+2^{g-1}=2^{8k-1}+2^{4k-1}=16^{2k-1}8+16^{k-1}8\equiv 8+8 \equiv 1\, \text{mod} \, 5.$$
  If $g\equiv 1 \,  \text{mod} \, 4$, let $g=4k+1$. Then 
$$2^{2g-1}+2^{g-1}=2^{8k+1}+2^{4k}=16^{2k}2+16^{k}\equiv 2+1= 3\, \text{mod} \, 5.$$
  If $g\equiv 3 \,  \text{mod} \, 4$, let $g=4k+3$. Then 
$$2^{2g-1}+2^{g-1}=2^{8k+5}+2^{4k+2}=16^{2k+1}2+16^{k}4\equiv 2+4 \equiv 1\, \text{mod} \, 5.$$
 
That is to say $5$ is not a divisor of $2^{2g-1}+2^{g-1}$ in each case.

(iii) By Lemma \ref{many primes} and its proof, there are infinite primes of the form $8k+7$, including 7, 
are not divisors of any $2^{2g-1}+2^{g-1}$.

In each of the three cases above,  there is a bounded $f$-invariant spin structure by Proposition \ref{B-U} (1) and then $f$ is extendable by Theorem \ref{criterion}.

Moreover, by Theorem \ref{4g+2,4g} (1) for $g=1$, the order 3 map $w_1^2$ is not extendable, and for $g=2$ the order 5 map $w_2^2$
is not extendable, where $w_g$ is a Wiman map on $F_g$. So the condition on $g$ in (i) and (ii) can not be removed in general. 
\end{proof}

%\begin{remark} Similar to the proof of Theorem *, one can verify that for each map $f: F_g\to F_g$ of order $p$
%$f$ has an unbounded invariant spin structure if either $p=3^m$ and $g$ is odd or $p=5^m$ and $g\equiv  1, 2, 3 \,  \text{mod} \, 4$. %In particular there are both bounded and unbounded $f$-invariant spin structures for odd $g$.
 %\end{remark}

%\section{More on $f$-invariant spin structures}

%\bibliographystyle{amsalpha}

\end{document}